\documentclass{amsart}
\usepackage{amsmath,amssymb,amsthm,amscd}
\newtheorem{theorem}{Theorem}[section]
\newtheorem{lemma}[theorem]{Lemma}
\newtheorem{proposition}[theorem]{Proposition}
\newtheorem{corollary}[theorem]{Corollary}

\newtheorem{example}[theorem]{Example}
\newtheorem{question}{Question}

\begin{document}

\title[indecomposable modules are pure-projective or pure-injective]{RINGS WHOSE INDECOMPOSABLE MODULES ARE
PURE-PROJECTIVE OR PURE-INJECTIVE}

\author{Fran\c{c}ois Couchot}
\address{Normandie Univ, UNICAEN, CNRS, LMNO, 14000 Caen, France}
\email{francois.couchot@yahoo.com} 



\begin{abstract}
Let $\mathcal{P}$ be the class of  rings for which every indecomposable right module is pure-projective or pure-injective. When $R$ is a Noetherian local commutative ring of maximal ideal $P$, it is proven that  $R\in\mathcal{P}$ if and only if $R$ is either an artinian valuation ring or  a discrete valuation domain of rank one with rank($\widetilde{R}$)$\leq 2$ where $\widetilde{R}$ is the completion of $R$ in its $P$-adic topology. Let $R$ be a commutative ring. Then $R\in\mathcal{P}$ if and only if $R$ is a clean arithmetical ring with $R_P\in\mathcal{P}$ for each maximal ideal $P$ of $R$. Moreover, $R$ is a semi-perfect ring when it is Noetherian. Some examples of commutative rings of the class $\mathcal{P}$ are given. 
\end{abstract}

\maketitle

In this paper we give a partial answer to the following question posed by D. 
Simson in \cite[Problem 3.2]{Sim98}:

\textit{”Give a characterization of rings $R$ for which every indecomposable right $R$-module is pure-projective or pure-injective. Is every such a semi-perfect ring $R$ right Artinian or right pure-semisimple?”}

A first partial answer to this question is \cite[Theorem 4.2]{GuSi01} where Guil Asensio and Simson proved the following: \textit{let $\mathcal{A}$ be a locally finitely presented Grothendieck category. Then $\mathcal{A}$ is pure-semisimple if and only if every indecomposable object of $\mathcal{A}$ is pure-projective.} 

So, we have the following theorem:
\begin{theorem}
Let $R$ be a ring for which each indecomposable right module is pure-projective. Then $R$ is right pure-semisimple.
\end{theorem}

We can also deduce this theorem from \cite[Proposition 1.13]{Azu92}, a result proven by Stenstr\"om: \textit{let $R$ be a ring and $M$ a right $R$-module. Then $M$ is pure-projective if and only if every pure submodule $N$ for which $M/N$ is indecomposable is a direct summand.} 

Let $\mathcal{P}$ be the class of  rings for which every indecomposable right module is pure-projective or pure-injective. In the sequel we study this question when $R$ is a commutative ring of the class $\mathcal{P}$. A complete characterization is given when $R$ is a locally Noetherian ring. In the general case, it remains to solve Question \ref{uniserial} about indecomposable modules over a maximal valuation ring. We show that $R$ is a clean ring and an elementary divisor ring if $R$ is a ring of the class $\mathcal{P}$. Moreover, for each maximal ideal $P$ of $R$, $R_P$ is an almost maximal valuation ring. So, if $R$ is a semi-perfect ring of the class $\mathcal{P}$, $R$ is not necessarily Artinian and some examples are given, but $R$ is pure-semisimple if $R$ is perfect. To get the main result of the first section a slight generalization of a result by Zanardo in \cite{Zan92} is proven and we use the properties of Gelfand rings to get the main result of the second section. Moreover, it is easy to see that the class $\mathcal{P}$ is preserved by Morita equivalence. This allows us to give some examples in the non-commutative case in the third section. 

Below, and in the sections too, some definitions are given but not all. The reader can refer to the following mathematical books \cite{AnFu92, Wis91, FuSa85, FuSa01}. All rings are associative and unitary, and all modules are unitary. A short exact sequence $\Sigma$ of right $R$-modules is \textit{pure-exact} if it remains exact when tensoring it with any left $R$-module. A right $R$-module $F$ is \textit{pure-injective} if for every pure-exact sequence $\Sigma$ of right $R$-modules, the sequence $\mathrm{Hom}_R(\Sigma,F)$ is exact. A right $R$-module $F$ is \textit{pure-projective} if for every pure-exact sequence $\Sigma$ of right $R$-modules the sequence $\mathrm{Hom}_R(F,\Sigma)$ is exact. Recall that a ring is said to be  right {\it pure-semisimple} if each right module is pure-injective.

When $R$ is a commutative ring we denote respectively $\mathrm{Spec}\ R$ and $\mathrm{Max}\ R$ the space of prime ideals and maximal ideals of $R$ with the Zariski topology. If $A$ a subset of $R$, then we denote  
\[V(A) = \{ P\in\mathrm{Spec}\ R\mid A\subseteq P\},\ D(A) = \{ P\in\mathrm{Spec}\ R \mid A\not\subseteq P\},\]
\[V_M(A)=V(A)\cap\mathrm{Max}\ R\ \mathrm{and}\ D_M(A) =D(A)\cap\mathrm{Max}\ R.\]

\section{Indecomposable modules are pure-projective or pure-injective: local case}

The following lemma will be useful in the sequel. 
\begin{lemma}\label{L:fp}
Let $R$ be a (non necessarily commutative) ring and $M$ a left $R$-module. Assume that $M$ verifies one of the following three conditions:
\begin{enumerate}
\item $M$ is finitely generated;
\item $M$ is indecomposable and there exists a pure-exact sequence of left $R$-modules $0\rightarrow K\rightarrow \bigoplus_{i\in I} F_i\rightarrow M\rightarrow 0$, where $F_i$ is a finitely generated $R$-module with a local endomorphism ring, for each $i\in I$;
\item $M$ is indecomposable and $M=\bigcup_{i\in I}M_i$ where $M_i$ is a finitely generated submodule of $M$ with a local endomorphism ring for each $i\in I$.
\end{enumerate}
Then $M$ is pure-projective if and only if it is finitely presented.
\end{lemma}
\begin{proof}
We only have to prove that $M$ is finitely presented if it is pure-projective. By \cite[34.1]{Wis91} $M$ is a direct summand of a direct sum of finitely presented left $R$-modules. 

$(1)$. If $M$ is finitely generated, it is a direct summand of a finite direct sum of finitely presented left $R$-modules. Hence $M$ is finitely presented. 

$(2)$. If $M$ satisfies the second condition, from the Krull-Schmidt theorem and the fact that $M$ is indecomposable we deduce that $M\cong F_i$ for some $i\in I$. Hence $M$ is finitely generated and from $(1)$ we deduce that it is finitely presented.

$(3)$ follows from $(2)$ because the sequence $0\rightarrow\ker{u}\rightarrow\bigoplus_{i\in I}M_i\xrightarrow{u}M\rightarrow 0$ is pure-exact, where $u$ is induced by the family of inclusion maps $(M_i\rightarrow M)_{i\in I}$ (see \cite[33.9.(2)]{Wis91}).
\end{proof}

Recall that a left (or right) module over a ring $R$ is {\it uniserial} if its set of submodules is totally ordered by inclusion. A commutative ring $R$ is a {\it valuation ring} if it is uniserial as $R$-module. A uniserial module $U$ is {\it linearly compact} (in its discrete topology) if every totally ordered family of cosets $(u_i+U_i)_{i\in I}$ has a nonempty intersection. A valuation ring $R$ is {\it maximal} if it is a linearly compact module and $R$ is {\it almost maximal} if $R/A$ is maximal for each nonzero ideal $A$.

\begin{proposition}\label{P:maximal}
For any valuation ring $R$ the following conditions are equivalent:
\begin{enumerate}
\item $R$ is maximal;
\item $R$ is a pure-injective $R$-module;
\item each uniserial $R$-module is pure-injective;
\item each uniserial $R$-module is linearly compact.
\end{enumerate} 
\end{proposition}
\begin{proof}
It is obvious that $(4)\Rightarrow (1)$ and $(3)\Rightarrow (2)$. By \cite[Proposition 4]{Couc06} $(1)\Leftrightarrow (2)$ and by \cite[Propositions 5]{Couc06} $(2)\Rightarrow (3)$.

$(1)\Rightarrow (4)$. Let $U$ be a uniserial $R$-module. If $U$ is finitely generated then $U=Ru\cong R/(0:u)$. In this case we easily conclude that $U$ is linearly compact. If $U$ is not finitely generated, let $(u_i+U_i)_{i\in I}$ be a totally ordered family of cosets of $U$. The implication is obvious if $U_i=U$ for each $i\in I$. Now assume there exists $i_0\in I$ such that $U_{i_0}\ne U$. In this case there exists $v\in U$ such that $u_{i_0}+U_{i_0}\subseteq Rv$ and for each $i\in I$ such that $U_i\subseteq U_{i_0}$ we also have $u_i+U_i\subseteq Rv$. Since $Rv$ is linearly compact we have $\bigcap_{i\in I}u_i+U_i\ne\emptyset$. Hence $U$ is linearly compact.
\end{proof}

Let $R$ be a valuation ring and $P$ its maximal ideal. Assume there exists a nonzero prime ideal $L$ with $L\ne P$. Then for each $s\in P\setminus L$ we have $L=sL$. By Nakayama Lemma $L$ is not finitely generated and $R$ is not Noetherian. So, if $R$ is Noetherian either $R$ is Artinian or $R$ is a {\it rank one discrete} valuation domain (we use the terminology of \cite{FuSa01}). In this last case, if $\widetilde{R}$ is the completion of $R$ in its $P$-adic topology, $\widetilde{R}$ is the pure-injective hull of $R$ (See \cite[Proposition XII.5.1]{FuSa01}). Moreover, when $\widetilde{R}$ is of finite rank, $R$ is called a {\it Nagata} valuation domain in \cite{Zan92}.

In \cite[Theorem]{Zan92} Zanardo used Kurosh invariants to show that each indecomposable torsionfree module $M$ of finite rank verifies rank($M$)$\leq 2$ if $R$ satisfies the assumptions of the following theorem. Another proof is done below and it is also shown that there is no indecomposable torsionfree module of infinite rank. Let $R$ be a valuation domain, $M$ an $R$-module and $N$ a submodule of $M$. Let $N'$ be the inverse image of the torsion submodule of $M/N$ by the natural map $M\rightarrow M/N$. It is easy to show that $M/N'$ is flat, whence $N'$ is a pure submodule of $M$. It is called the {\it purification} of $N$ in $M$ (\cite[I.8]{FuSa01}). When $M$ is torsionfree, rank($N'$)$=$rank($N$).

\begin{theorem}\label{T:torsionfree}
Let $R$ be a discrete valuation domain of rank $1$, $P=pR$ its maximal ideal and $Q$ its quotient field. Assume that rank($\widetilde{R}$)$=2$. Then each indecomposable torsionfree $R$-module is isomorphic to one of the modules $R$, $Q$, $\widetilde{R}$. 
\end{theorem}
\begin{proof}
Let $M$ be an indecomposable torsionfree $R$-module. There exists an exact sequence $0\rightarrow F\rightarrow M\rightarrow D\rightarrow 0$, where $F$ is a free module, $D$ a $Q$-vector space and $F$ satisfies the following conditions: $F$ is a pure submodule of $M$ and $F/PF\cong M/PM$ (for example see \cite[Proposition 21]{Coucho07}). If rank($M$)$=1$ either $M=F\cong R$ or $M=D\cong Q$, and if rank($M$)$=2$ we will show that $M\cong\widetilde{R}$. In this case rank($F$)$=1$ and rank($D$)$=1$. So, there exists an isomorphim $\alpha:F\rightarrow R$ which extends to a homomorphism $\beta:M\rightarrow\widetilde{R}$ which induces a homomorphism $\gamma:D\rightarrow\widetilde{R}/R$. We have the following commutative diagram:

\[\begin{matrix}
{} & {} &\ \ 0 & {} & {} & {} & {} & {} & {}\\
{} & {} &\ \ \downarrow & {} & {} & {} & {} & {} & {} \\
0 & \rightarrow & \ \ F & \rightarrow & \ \ M & \rightarrow & \ \ D & \rightarrow & 0 \\
 {} & {} & {\scriptstyle{\alpha}}\downarrow & {} & {\scriptstyle{\beta}}\downarrow & {} & {\scriptstyle{\gamma}}\downarrow & {} & {} \\
0 & \rightarrow & \ \ R & \rightarrow & \ \ \widetilde{R} & \rightarrow & \ \ \widetilde{R}/R & \rightarrow & 0 \\
{} & {} &\ \ \downarrow & {} & {} & {} & {} & {} & {} \\
{} & {} &\ \ 0 & {} & {} & {} & {} & {} & {} 
\end{matrix}\]

Since $D$ and $\widetilde{R}/R$ are $Q$-vector spaces of dimension $1$, either $\gamma $ is an isomorphism and so is $\beta$, or $\gamma=0$ and in this case we get a contradiction because by Snake Lemma $\ker(\beta)\cong D$ and it is a direct summand of $M$. By way of contradiction assume that rank($M$)$>2$. Since $M$ is indecomposable we have $F\ne 0$ and $D\ne 0$. Let $x\in M\setminus F$ and $U$ be the purification of $Rx$ in $M$. Since $U$ is of rank $1$ and $M$ is indecomposable, $U$ is isomorphic to $R$. So, if $u$ generates $U$ then $u\notin F$. The equality $M=F+PM$ implies that there exist $y\in F$ and $z\in M$ such that $u=y+pz$. Clearly $u\notin pU$ and consequently $u\notin pM$ because $U$ is a pure submodule. It follows that $y\notin pM$. If $y\in U$, the purity of $U$ in $M$ implies that $(1-pr)u=y$ for some $r\in R$ and we get that $u\in F$, a contradiction. So, $y\notin U$. Let $V$ be the purification of $U+Ry$ in $M$. The indecomposability of $M$ implies that $V$ is free of rank $2$. Clearly $u\notin pV$ and $y\notin pV$. Assume $au=by$ for some $a,b\in R$. Then $u\notin F$ implies $a\in P$ and $y\notin U$ implies $b\in P$. So, if $\bar{u}, \bar{y}$ are the respective images of $u,y$ in $V/pV$, $\{\bar{u},\bar{y}\}$ is a basis of $V/pV$. By Nakayama Lemma $\{u,y\}$ is a basis of $V$. From $V$ a pure submodule of $M$ we deduce that $u=y+psu+pty$ for some $s,t\in R$. We get a contradiction. Hence rank($M$)$\leq 2$.
\end{proof}

A commutative local ring $R$ is said to be {\it Henselian} if each commutative module-finite $R$-algebra is a finite product of local rings.

\begin{theorem}\label{T:main}
Let $R$ be a commutative local ring of maximal ideal $P$. Consider the following conditions:
\begin{enumerate}
\item each indecomposable $R$-module is pure-projective or pure-injective;
\item $R$ is a valuation ring;
\item $R$ is a maximal valuation ring and each indecomposable $R$-module is uniserial;
\item either $R$ is pure-semisimple or $R$ is a discrete valuation domain of rank $1$ with rank($\widetilde{R}$)$\leq 2$;
\item each indecomposable $R$-module is pure-injective;
\end{enumerate}
Then the following assertions hold:
\begin{itemize}
\item $(1)\Rightarrow (2)$;
\item if $R$ is not Noetherian then $(1)\Leftrightarrow (3)\Leftrightarrow (5)$;
\item if $R$ is Noetherian then $(1)\Leftrightarrow (4)$;
\item $(1)$ implies that $R$ is Henselian.
\end{itemize}
\end{theorem}
\begin{proof}
$(1)\Rightarrow (2)$. By way of contradiction suppose that $R$ is not a valuation ring. As in the proof of \cite[Lemma III.3]{Cou07} assume that there exist $a,\ b\in P$ such that $a\notin Rb$ and $b\notin Ra$. Let $R'=R/(Pa+Pb)$ and $M$ be the $R'$-module generated by $(x_n)_{n\in\mathbb{N}}$ with the relations $(ax_n=bx_{n+1})_{n\in\mathbb{N}}$. It is proven that $M$ is indecomposable and End$_R(M)$ is not local. So, $M$ is not pure-injective by \cite[Theorem 9(3)]{ZHZ78}. For each positive integer $n$ let $M_n$ be the submodule of $M$ generated by $\{x_m\mid 0\leq m\leq n\}$ ($M_n$ is finitely presented over $R'$ but not necessarily over $R$). By \cite[Propositions 3.3 and 3.1]{Cou11} End$_R(M_n)$ ($=$End$_{R'}(M_n)$) is local ($M_n\cong D(W_{1,n-1,n})$). Clearly $M=\bigcup_{n\in\mathbb{N}}M_n$ and consequently from Lemma \ref{L:fp} we deduce that $M$ is not pure-projective. Hence $R$ is a valuation ring.

The implication $(5)\Rightarrow (1)$ is obvious.

$(1)\Rightarrow (3)$ and $(5)$. Let $U$ be a uniserial $R$-module. For each nonzero $u\in U$ we have that End($Ru$)$=R/(0:u)$ is a local ring. So, if $U$ is not finitely presented then we deduce from Lemma \ref{L:fp} that $U$ is pure-injective. Let $A$ be a non-finitely generated ideal of $R$. Then, by Proposition \ref{P:maximal} $A$ and $R/A$ are linearly compact. By \cite[Proposition 9]{Zel53} $R$ is also linearly compact. Hence $R$ is maximal. Let $M$ be an indecomposable module. There exists a pure-exact sequence of $R$-modules $0\rightarrow K\rightarrow \bigoplus_{i\in I} F_i\rightarrow M\rightarrow 0$, where $F_i$ is a cyclic $R$-module for each $i\in I$. Then End$_R$($F_i$) is local for each $i\in I$. So, if $M$ is pure-projective, it is finitely presented by Lemma \ref{L:fp} and cyclic; it is also pure-injective because $R$ is maximal. Hence each indecomposable $R$-module $M$ is pure-injective and by \cite[Theorem 5.4]{Fac87} $M$ is either the pure-injective hull of a uniserial module or an essential extension of a cyclic module. Since $R$ is maximal, $M$ is uniserial by \cite[Theorem]{Gil71}.

$(3)\Rightarrow (1)$ Let $M$ be an indecomposable $R$-module. Since $R$ is a maximal valuation ring and $M$ a uniserial $R$-module, $M$ is pure-injective by Proposition \ref{P:maximal}. 

$(1)\Rightarrow (4)$. Since $R$ is a Noetherian valuation ring, either $R$ is Artinian, in this case $R$ is pure-semisimple, or $R$ is a discrete valuation domain of rank $1$. If $R\ne\widetilde{R}$ let $x\in\widetilde{R}\setminus R$ and $U$ be the purification of $R+Rx$ in $\widetilde{R}$. By \cite[Example XV.6.1.]{FuSa01} $U$ is indecomposable and it is a flat module of rank $2$. So $U$ is not free and consequently it is not pure-projective. Hence $U$ is pure-injective, $U=\widetilde{R}$ and rank($\widetilde{R}$)$\leq 2$.

$(4)\Rightarrow (1)$. It is obvious if  $R$ is not a domain. Assume that $R$ is a domain and rank($\widetilde{R}$)$\leq 2$. Let $M$ be an indecomposable $R$-module. If its torsion submodule $^tM$ is nonzero there exists a uniserial module $U\ne 0$ which is a pure submodule of $^tM$ and of $M$ too by \cite[Theorem XI.6.11]{FuSa01}. But by \cite[Theorem XIII.4.6]{FuSa01} $U$ is pure-injective. We conclude that $U=M$. If $M$ is torsionfree, by Theorem \ref{T:torsionfree}  $M$ is isomorphic to one of the  $R$-modules $R$, $Q$, $\widetilde{R}$.

Each maximal valuation ring $R$ is Henselian by \cite[Theorem II.7.7]{FuSa01} even if $R$ is not an integral domain. If $R$ is a discrete valuation domain of rank one, the property rank $\widetilde{R}\leq 2$ implies that $R$ is Henselian too by \cite[Proposition 1]{Rib64}.
\end{proof}

\begin{question}\label{uniserial}
\textnormal{What are the non-Noetherian maximal valuation rings whose indecomposable modules are uniserial?}
\end{question}

Now we can give an example of a non-Artinian commutative semi-perfect ring whose indecomposable modules are pure-injective.

\begin{example}\label{E:local1}
\textnormal{Let $R$ be a complete discrete valuation domain of rank $1$ . By Theorem \ref{T:main} each indecomposable module is pure-injective. Clearly $R$ is semi-perfect but not Artinian.}
\end{example}

\begin{corollary}
Let $R$ be a commutative local ring of maximal $P$. Then the following conditions are equivalent:
\begin{enumerate}
\item each indecomposable $R$-module is pure-projective or pure-injective and there exist indecomposable $R$-modules which are not pure-injective;
\item $R$ is a discrete valuation domain of rank $1$ with rank($\widetilde{R}$)$=2$.
\end{enumerate} 
Each indecomposable $R$-module which is not pure-injective is isomorphic to $R$ when these conditions hold.
\end{corollary}

\begin{example}\label{E:local}
\textnormal{There is at least one example of a local ring $R$ satisfying the conditions of the previous corollary: the one of characteristic $2$ given by Nagata in \cite[Example E33, p.207]{Nag62}.}
\end{example}

\begin{corollary}
Let $R$ be a commutative perfect ring for which each indecomposable module is pure-projective or pure-injective. Then $R$ is pure-semisimple.
\end{corollary}
\begin{proof} $R$ is a finite product of local rings which are Artinian valuation rings by Theorem~\ref{T:main}. Hence $R$ is pure-semisimple. 
\end{proof}

\begin{question}\label{perfect} \textnormal{Let $R$ be a non-commutative ring. Assume that $R$ is left (or right) perfect. Suppose that each indecomposable left $R$-module is pure-projective or pure-injective. Is $R$ left (or right) pure-semisimple?}
\end{question}

\section{Indecomposable modules are pure-projective or pure-injective: global case}

The following proposition will be useful in the sequel.

\begin{proposition}\label{P:epi}
Let $R$ and $R'$ be  commutative rings, $\alpha: R\rightarrow R'$ a ring epimorphism (in the categorical sense). We consider an $R'$-module $M$. The following assertions hold:
\begin{enumerate}
\item $M$ is indecomposable over $R$ if and only if it is indecomposable over $R'$;
\item $M$ is pure-injective over $R$ if and only if it is pure-injective over $R'$;
\item if $M$ is pure-projective over $R$ then $M$ is pure-projective over $R'$ too. 
\end{enumerate}
\end{proposition}
\begin{proof}
$(1)$. Assume that $M$ is indecomposable over $R'$. Suppose that $M= K\oplus L$ where $K$ and $L$ are $R$-submodules of $M$. Then \[M\cong R'\otimes_RM= (R'\otimes_RK)\oplus (R'\otimes_RL)\] and we deduce that either $K\subseteq R'\otimes_RK=0$ or $L\subseteq R'\otimes_RL=0$. The converse is obvious.

$(2)$. Assume that $M$ is pure-injective over $R'$. If $\Sigma$ is a pure-exact short sequence of $R$-modules then $R'\otimes_R\Sigma$ is a pure-exact sequence of $R'$-modules and\\ Hom$_R(\Sigma,M)\cong$Hom$_{R'}(R'\otimes_R\Sigma,M)$. So, $M$ is pure-injective over $R$. Conversely, if $\Sigma$ is a pure-exact short sequence of $R'$-modules, $\Sigma$ is also pure-exact over $R$. It follows that $M$ is pure-injective over $R'$ too.

$(3)$. There exists a pure-exact sequence of $R'$-modules $0\rightarrow K\rightarrow L\rightarrow M\rightarrow 0$ where $L$ is pure-projective over $R'$. This sequence is also pure-exact over $R$. We deduce that $M$ is isomorphic (over $R$ and $R'$) to a direct summand of $L$.
\end{proof}

A commutative ring (integral domain) $R$ is said to be {\it arithmetical} ({\it Pr\"ufer}) if $R_P$ is a valuation ring (domain) for each maximal ideal $P$.

\begin{theorem}\label{T:domain}
Let $R$ be an integral domain. Then the following conditions are equivalent:
\begin{enumerate}
\item each indecomposable $R$-module is pure-projective or pure-injective;
\item $R$ verifies one of these assertions:
\begin{enumerate}
\item $R$ is a maximal valuation domain and each indecomposable module is uniserial;
\item $R$ is a discrete valuation domain of rank $1$ with rank($\widetilde{R}$)$\leq 2$.
\end{enumerate}
\end{enumerate}
\end{theorem}
\begin{proof}
$(2)\Rightarrow (1)$ follows from Theorem \ref{T:main}.

$(1)\Rightarrow (2)$. By Proposition \ref{P:epi} $R_P$ satisfies the condition $(1)$ for each maximal ideal $P$ and it is a valuation domain by Theorem \ref{T:main}. Consequently $R$ is a Pr\"ufer domain. So, $R$ is integrally closed by \cite[Corollary I.3.5]{FuSa01}. Now, we will show that $R$ is local.

Let $P$ be a finitely generated maximal ideal. By way of contradiction suppose that $R$ is not local. So, there is another maximal ideal $P'$. Let $R'=S^{-1}R$ where $S=R\setminus (P\cup P')$. By Proposition \ref{P:epi} $R'$ also verifies the condition $(1)$. So we may replace $R$ with $R'$. Since $R$ is a semilocal Pr\"ufer domain, each finitely generated ideal of $R$ is principal by \cite[Theorem 5]{Jen66}. Now, we do as in the proof of \cite[Theorem 63]{Mat72}. We have $P=Rp$. Let $p'\in P'\setminus P$. Clearly $p+p'$ is a unit of $R$. We may assume that $p+p'=1$. Let $n$ be a prime integer different to the characteristic of $R/P'$. Consider the polynomial $(X^n-p)$ in $R[X]$. Its image in $(R/P')[X]$ is $(X^n-1)$ and $1$ is a simple root of this polynomial. By Proposition \ref{P:epi} $R_{P'}$  verifies the condition $(1)$. It is Henselian by Theorem \ref{T:main}. By \cite[Theorem II.7.3]{FuSa01} there exists $s\in R_{P'}$ such that $p=s^n$. Since $R$ is integrally closed  $s\in P$. This contradicts that $p$ generates $P$. So, $R$ is local in this case.

Now let $P$ be a non finitely generated maximal ideal of $R$. Then $P$ is a flat module of rank $1$. Hence it is not pure-projective. So, it is pure-injective and its endomorphism ring is local by \cite[Theorem 9(3)]{ZHZ78}. We will show that End$_R(P)=R$. Let $Q$ be the quotient field of $R$. Let $A$ and $B$ be submodules of $Q$. We have $\mathrm{Hom}_R(A,B)\cong\{q\in Q\mid qA\subseteq B\}.$ As in \cite[I.2]{FuSa01} we denote this last set by $B:A$. Clearly $R\subseteq P:P\subseteq R:P$. Since $P$ is not finitely generated we have $P\subset R:(R:P)\subseteq R:R=R$. So, $R:(R:P)=R$ because $P$ is a maximal ideal. Hence $P:P=R:P=R.$

Consequently $R$ is local and we conclude by Theorem \ref{T:main}.   
\end{proof}

A ring $R$ is called {\it Gelfand} if every prime ideal of $R$ is contained in only one maximal ideal. In the following proposition there are some (already known) properties of Gelfand rings useful to show Proposition \ref{P:indecomposable} and Theorem \ref{T:main1} (see for instance \cite{BoVa83} the book by Borceux and Van den Bossche). If $L$ is a prime ideal we denote by $\mu(L)$ the unique maximal ideal containing $L$. We set  $0_P$  the kernel of the natural map $R\rightarrow R_P$ where $P\in\mathrm{Spec}\ R$.

\begin{proposition}
\label{P:Gelfand} Let $R$ be a commutative Gelfand ring. Then $R$ satisfies the following properties:
\begin{enumerate}
\item $\mu: \mathrm{Spec}\ R\rightarrow \mathrm{Max}\ R$ is continuous and $\mathrm{Max}\ R$ is Hausdorff;
\item a closed subset $C$ of $\mathrm{Spec}\ R$ is the inverse image of a closed subset of $\mathrm{Max}\ R$ by $\mu$ if and only if $C=V(A)$ where $A$ is a pure ideal. Moreover, in this case, $A=\bigcap_{P\in C}0_P$.
\item for each maximal ideal $P$, $R_P=R/0_P$;
\item a subset $U$ of $\mathrm{Max}\ R$ is open and closed if and only if there exists an idempotent $e\in R$ such that $U=D_M(e)$;
\item $R$ is indecomposable if and only if $\mathrm{Max}\ R$ is connected.
\end{enumerate}
\end{proposition}
\begin{proof} 
$(1)$. See \cite[Theorem 1.2]{MaOr71}.

$(2)$.
Let $A$ be a pure ideal, and let $P$ be a maximal ideal and $L$ a prime ideal such that $A\subseteq P$ and $L\subseteq P$. Since $A$ is pure, for each $a\in A$ there exists $b\in A$ such that $a=ab$. Then $(1-b)a=0$ and $(1-b)\notin P$, whence  $(1-b)\notin L$ and $a\in L$. So, $L\in V(A)$ and $V(A)$ is the inverse image of a closed subset of $ \mathrm{Max}\ R$ by $\mu$.

Let $C=V(B)$ where $B=\bigcap_{L\in C}L$. Suppose that $C$ is the inverse image of a closed subset of $\mathrm{Max}\ R$ by $\mu$. We put $A=\bigcap_{P\in C}0_P$. Let $b\in B$ and $P\in C$. Then $C$ contains each minimal prime ideal contained in $P$. So, the image of $b$, by the natural map $R\rightarrow R_P$, belongs to the nilradical of $R_P$. It follows that there exist  $0\ne n_P\in\mathbb{N}$ and $s_P\in R\setminus P$ such that $s_Pb^{n_P}=0$. Hence, $\forall L\in D(s_P)\cap C,\ b^{n_P}\in 0_L$. A finite family $(D(s_{P_j}))_{1\leq j\leq m}$ covers $C$. Let $n=\max\{n_{P_1},\dots,n_{P_m}\}$. Then $b^n\in 0_L,\ \forall L\in C$, whence $b^n\in A$. We deduce that $C=V(A)$. Now, we have $A_P=0$ if $P\in V(A)$ and $A_P=R_P$ if $P\in D(A)$. Hence $A$ is a pure ideal.

$(3)$. If $C$ is the inverse image of $\{P\}$ by $\mu$ then $C=V(0_P)$ and $R/0_P$ is a local ring and a flat $R$-module. So, $R_P=R/0_P$.

$(4)$. A subset $U$ of $\mathrm{Max}\ R$ is open and closed if and only if is so its inverse image by $\mu$ and it is well known that a subset $U'$ of $\mathrm{Spec}\ R$ is open and closed if and only if $U'=D(e)$ for some idempotent $e\in R$.

$(5)$ is an immediate consequence of $(4)$.
\end{proof} 

A topological space is called \textit{totally disconnected} if each of its connected components contain only one point. Every Hausdorff topological space $X$ with a base of clopen (closed and open) neighbourhoods is totally disconnected and the converse holds if $X$ is compact (see \cite[29.7 Theorem]{Wil04}). A point $x$ of a topological space $X$ is {\it isolated} if $\{x\}$ is an open subset of $X$.

\begin{lemma}\label{L:topospace}
Let $X\ne\emptyset$ be a connected Hausdorff compact  topological space. Assume that each non-empty proper closed subspace $F$ of $X$ is connected if and only if $F=\{x\}$ for some $x\in X$. Then $X$ contains only one point.
\end{lemma}
\begin{proof}
By way of contradiction suppose there exist $x, y\in X$ with $x\ne y$. Since $X$ is Hausdorff there exist two disjoint open subsets $U,\ V$ of $X$ such that $x\in U$ and $y\in V$. Then the closure $\overline{U}$ of $U$ in $X$ doesn't contain $y$. So, $\overline{U}\ne X$ and $U\ne\overline{U}$ because $X$ is connected. We put $Z= \overline{U}\cap (X\setminus U)$. Since $\overline{U}$ is compact, the property satisfied by $X$ implies that $\overline{U}$ is totally disconnected. So, for each $z\in Z$ there exist two disjoint clopen subsets $U_z, V_z$ of $\overline{U}$ such that $x\in U_z$ and $z\in V_z$. But, $Z$ is closed. So, it is compact and there exists a finite subset $F$ of $Z$ such that $Z\subseteq \bigcup_{z\in F}V_z$. We put $W=\bigcap_{z\in F}U_z$. We get that $W$ is a nonempty clopen subset of $\overline{U}$ and $W\subseteq U$. There exist a closed subset $G$ and an open subset $O$ of $X$ such that $W= \overline{U}\cap G= \overline{U}\cap O$. It follows that $W$ is a closed subset of $X$. But $W\subseteq U\cap O\subseteq \overline{U}\cap O=W$. So, $W$ is a proper clopen subset of $X$. This contradicts that $X$ is connected.
\end{proof}

\begin{proposition}\label{P:indecomposable}
Let $R$ be a commutative ring. Assume that each indecomposable $R$-module is pure-projective or pure-injective. Then $R$ verifies the following conditions:
\begin{enumerate}
\item $R$ is a Gelfand ring;
\item $R$ is a local ring if $R$ is an indecomposable ring.
\end{enumerate}
\end{proposition}
\begin{proof}
$(1)$. Let $L$ be a prime ideal of $R$. Then $R/L$ is an integral domain satisfying $(1)$ by Proposition \ref{P:epi}. From Theorem \ref{T:domain} we deduce that $R/L$ is local. 

$(2)$. Since $R$ is indecomposable $\mathrm{Max}\ R$ is a connected Hausdorff compact topological space by Proposition \ref{P:Gelfand}. Let $F$ be a nonempty proper connected closed subset of $\mathrm{Max}\ R$. Then there exists a pure ideal $B$ of $R$ such that $F=V_M(B)$, and $B$ is not finitely generated: else $B=Re$ for some idempotent of $R$ and $\mathrm{Max}\ R$ is not connected. It follows that $R/B$ is an indecomposable finitely generated $R$-module which is not pure-projective by Lemma \ref{L:fp}. Consequently $R/B$ is a pure-injective module and a local ring by \cite[Theorem 9(3)]{ZHZ78}. So, $F$ contains only one point and $\mathrm{Max}\ R$ satisfies the assumptions of Lemma \ref{L:topospace}. We conclude that $\mathrm{Max}\ R$ contains only one point and $R$ is local.
\end{proof}

A ring $R$ is {\it clean} if each of its elements is the sum of a unit with an idempotent. Recall that $R$ is an \textit{elementary divisor ring} if each finitely presented module is a direct sum of cyclic submodules.

\begin{theorem}\label{T:main1}
Let $R$ be a commutative ring. Then the following conditions are equivalent: 
\begin{enumerate}
\item each indecomposable $R$-module is pure-projective or pure-injective;
\item $R$ is a clean ring and for each maximal ideal $P$, $R_P$ verifies one of the following two conditions:
\begin{enumerate}
\item $R_P$ is a maximal valuation ring and each indecomposable $R_P$-module is uniserial;
\item $R_P$ is a discrete valuation domain of rank $1$ with rank($\widetilde{R_P}$)$\leq 2$;
\end{enumerate}
\end{enumerate} 
Moreover $R$ is an elementary divisor ring if $R$ satisfies these conditions.
\end{theorem}
\begin{proof}
$(1)\Rightarrow (2)$.
By \cite[Theorem I.1]{Cou07}, to prove that $R$ is clean we have to show that $\mathrm{Max}\ R$ is totally disconnected. Let $C\ne\emptyset$ be a connected component of $\mathrm{Max}\ R$. Since $C$ is closed $C=V_M(A)$ for some ideal $A$ of $R$ and $C$ is homeomorphic to Max $R'$ where $R'=R/A$. By Proposition \ref{P:epi} $R'$ satisfies condition $(1)$. So, $R'$ is indecomposable by Proposition \ref{P:Gelfand} and local by Proposition \ref{P:indecomposable}. It follows that $C$ contains only one point and $\mathrm{Max}\ R$ is totally disconnected. Hence $R$ is clean.

Let $P$ be a maximal ideal. By Proposition \ref{P:epi} $R_P$ satisfies the condition $(1)$. We conclude by Theorem \ref{T:main}.

$(2)\Rightarrow (1)$. Let $M$ be an indecomposable $R$-module. By \cite[Proposition III.1]{Cou07}, there exists a unique maximal ideal $P$ such that $M = M_P$ and its structures of $R$-module and $R_P$-module coincide. So, $M$ is pure-projective or pure-injective by Theorem \ref{T:main}.

The last assertion holds by \cite[Corollary II.3]{Cou07}.
\end{proof}

\begin{corollary}\label{C:main2}
Let $R$ be a commutative ring $R$ satisfying the conditions of Theorem \ref{T:main1}. Then the following assertions hold:
\begin{enumerate}
\item each indecomposable $R$-module is pure-injective if $R$ verifies one of the following two conditions:
\begin{enumerate}
\item $R_P$ is a maximal valuation ring for each maximal ideal $P$;
\item  there are no isolated points in $\mathrm{Max}\ R$.
\end{enumerate}
\item consequently, there exists an indecomposable projective $R$-module which is not pure-injective if there exists an isolated point $P$ of $\mathrm{Max}\ R$ such that $R_P$ is a discrete valuation domain of rank one satisfying rank($\widetilde{R_P}$)$=2$.
\end{enumerate}
\end{corollary}
\begin{proof}
$(1)$. It is obvious if $R$ satisfies $(a)$. Let $P$ be a non isolated point of $\mathrm{Max}\ R$. It follows that the pure ideal $0_P$ is not finitely generated. So, $R_P=R/0_P$ is not pure-projective by Lemma \ref{L:fp}, whence $R_P$ is maximal.

$(2)$ is an immediate consequence of $(1)$.
\end{proof}

A commutative ring $R$ is {\it locally Noetherian} if $R_P$ is Noetherian for each maximal ideal $P$.

\begin{corollary} \label{C:locnoe}
Let $R$ be a locally Noetherian ring. Then the following conditions are equivalent:
\begin{enumerate}
\item each indecomposable $R$-module is pure-projective or pure-injective; 
\item $R$ is a clean ring and for each maximal ideal $P$, $R_P$ is either an Artinian valuation ring or a discrete valuation domain of rank $1$ satisfying rank($\widetilde{R_P}$)$\leq 2$.
\end{enumerate}
\end{corollary}

\begin{example}\label{E:global1}
\textnormal{If $R$ is the ring defined in \cite[Example III.8]{Cou07} or if $R$ satisfies the equivalent conditions of \cite[Theorem III.4]{Cou07} then $R$ is a ring for which each indecomposable module is pure-injective and uniserial.}
\end{example}

\begin{corollary}
Let $R$ be a commutative Noetherian ring. Then the following conditions are equivalent:
\begin{enumerate}
\item each indecomposable $R$-module is pure-projective or pure-injective;
\item there exists a positive integer $n$ such that $R$ is the direct product of $n$ rings $(R_i)_{1\leq i\leq n}$, where $R_i$ is either an Artinian valuation ring or a discrete valuation domain of rank $1$ satisfying rank($\widetilde{R_i}$)$\leq 2$ for each $i,\ 1\leq i\leq n$.
\end{enumerate}
\end{corollary}
\begin{proof}
$(2)\Rightarrow (1)$ follows from Theorem \ref{T:main} because we have $M=\bigoplus_{1\leq i\leq n}R_i\otimes_RM$ for each $R$-module $M$.

$(1)\Rightarrow (2)$. Since $R$ is Noetherian the set of its minimal prime ideals is finite. These ideals are comaximal because $R$ is arithmetical. So, $R$ is a finite direct product of Noetherian valuation rings. We complete the proof with Corollary \ref{C:locnoe}.
\end{proof}

\begin{example}\label{E:global}
Let $D$ be a discrete valuation domain of rank $1$ and $J$ its maximal ideal. Assume that rank($\widetilde{D}$)$=2$ and $D$ contains a subring $K$ which is a field (for instance, the Nagata's example of characteristic $2$). Let $I$ be an infinite set. Consider the unitary $K$-subalgebra $R$ of $D^{I}$ generated by $D^{(I)}$. Then each indecomposable $R$-module is pure-projective or pure-injective and there exist indecomposable projective $R$-modules which are not pure-injective. Moreover, $R$ is locally Noetherian but not Noetherian.
\end{example}
\begin{proof}
For each $i\in I$, let $e_i=(\delta_{i,j})_{j\in I}$ where $\delta_{i,j}$ is the Kronecker symbol. Let $P$ be a maximal ideal of $R$. Either $e_i\in P$ for each $i\in I$ and in this case $P=\bigoplus_{i\in I}Re_i=D^{(I)}$ and $R_P=R/P\cong K$; or there exists $i\in I$ such that $e_i\notin P$, and in this case $P=Je_i+R(1-e_i)$ and $R_P=R/R(1-e_i)\cong D$. It is easy to prove that $R$ is a clean ring  by showing that each element of $R$ is the sum of a unit with an idempotent. We conclude by using Corollary \ref{C:locnoe}. For each $i\in I$, $Re_i$ is an indecomposable projective $R$-module which is not pure-injective.
\end{proof}

\section{Some examples in non-commutative case}

The main goal of this section is to give some examples of non-commutative rings for which each indecomposable right (or left) module is pure-projective or pure-injective. We will use the Morita equivalence between two rings. For this we refer to \cite[Chapter 6, Sections §21 and §22]{AnFu92} the book by Anderson and Fuller.

\begin{proposition}\label{P:equivalence}
Let $R$ and $S$ be two rings. Consider the following conditions:  
\begin{enumerate}
\item each indecomposable right $R$-module is pure-projective or pure-injective;
\item each indecomposable right $S$-module is pure-projective or pure-injective;
\item each indecomposable left $R$-module is pure-projective or pure-injective;
\item each indecomposable left $S$-module is pure-projective or pure-injective. 
\end{enumerate}
If $R$ and $S$ are Morita equivalent then $(1)\Leftrightarrow (2)$ and $(3)\Leftrightarrow (4)$.
\end{proposition}
\begin{proof}
We use the same terminology and the same notations as in \cite[Chapter 6, Sections §21 and §22]{AnFu92}. Let $F$ be the category equivalence $_R\mathbf{M}\rightarrow {}_S\mathbf{M}$ and $G$ the inverse equivalence of $F$. Recall that a short sequence $\Sigma$ of modules is pure-exact if and only if the sequence Hom$(K,\Sigma)$ is exact for each finitely presented module $K$. Clearly, if $K$ is a left $R$-module, $K$ is finitely presented if and only if $F(K)$ is finitely presented. If $K,\ L$ are left $R$-modules, $F$ induces a group isomorphism between Hom$_R(K,L)$ and Hom$_S(F(K),F(L))$. So, if $\Sigma$ is a short sequence of $R$-modules, $\Sigma$ is pure-exact if and only if $F(\Sigma)$ is pure-exact. Then, it is easy to show that, if $K$ is a left $R$-module, $K$ is pure-projective (resp. pure-injective) if and only if $F(K)$ is pure-projective (resp. pure-injective). On the other hand, it is well known that $K$ is indecomposable if and only if $F(K)$ is indecomposable. Now, we easily complete the proof.
\end{proof}

A ring $R$ is {\it called strongly clean} if each element of $R$ is the sum of an idempotent and a unit that commute.

\begin{corollary}
Let $R$ be a commutative ring and $n>1$ an integer. Let $R_n$ be the matrix ring $\mathbb{M}_n(R)$. Then the following conditions are equivalent:
\begin{enumerate}
\item each indecomposable $R$-module is pure-projective or pure-injective;
\item each indecomposable right $R_n$-module is pure-projective or pure-injective;
\item each indecomposable left $R_n$-module is pure-projective or pure-injective.
\end{enumerate}
Moreover, if these conditions are satisfied, $R_n$ is strongly clean.
\end{corollary}
\begin{proof}
We deduce the equivalence of the three conditions from \cite[22.3. Corollary]{AnFu92} and Proposition \ref{P:equivalence}.

$R$ is clean and by Theorem \ref{T:main} $R_P$ is Henselian for each maximal $P$. So, $R_n$ is strongly clean by \cite[Corollary 2]{Cou08}.
\end{proof}

\begin{example}
\textnormal{If $R$ is the ring of Example \ref{E:local1} or one of the rings of Example \ref{E:global1}, then for each integer $n>1$, each indecomposable left (resp. right) module over $\mathbb{M}_n(R)$ is pure-injective.}

\textnormal{If $R$ is the ring of Example \ref{E:local} or of Example \ref{E:global}, then for each integer $n>1$, each indecomposable left (resp. right) module over $\mathbb{M}_n(R)$ is pure-projective or pure-injective, and there exist indecomposable left (right) $\mathbb{M}_n(R)$-modules which are not pure-injective.}
\end{example}

Below we give an example of a non-commutative ring $R$ whose  indecomposable right modules are pure-injective (injective) and for which there exists an indecomposable left module which is neither pure-injective (injective) nor pure-projective (projective). So, $R$ is not Morita equivalent to a commutative ring.

\begin{example}\label{E:1}
\textnormal{Let $K$ be a field, $V$ a vector space over $K$ which is not of finite dimension, $S=\mathrm{End}_K(V)$, $J$  the set of finite rank elements of $S$ and $R$ be the unitary $K$-subalgebra of $S$ generated by $J$. Then, $R$ is von Neumann regular, $V$ and ${}_RR/J\cong K$ are the sole types of simple left modules, and, $V^*=\mathrm{Hom}_K(V,K)$ and $R_R/J\cong K$ are the sole types of simple right modules. It is easy to check that $V^*$ and $R_R/J$ are injective, whence $R$ is a right V-ring. On the other hand each right $R$-module contains a simple module. So, every indecomposable right module is simple. It follows that every indecomposable right module is pure-injective (injective). It is well known that $V$ is FP-injective,  $V^{**}$ is the injective hull of $V$ and $V\subset V^{**}$. So, $V$ is not pure-injective but it is pure-projective. We have $JV^{**}=V$. So, the structure of $R$-module of $V^{**}/V$ coincides with its one of $K$-vector space. On the other hand $V^{**}/V$ is of infinite rank over $K$. Let $M$ be the inverse image of a finite dimensional vector subspace of $V^{**}/V$ by the natural map $V^{**}\rightarrow V^{**}/V$. Then $M$ is indecomposable and it is finitely generated but not finitely presented, ($M$ finitely presented successively implies  $M/JM=M/V$ is finitely presented, $K\cong R/J$ is finitely presented over $R$ and $J$ is finitely generated that is not true). By Lemma \ref{L:fp} $M$ is not pure-projective, and since $M\subset V^{**}$, $M$ is not pure-injective too.}
\end{example}



\end{document}